\documentclass[a4paper]{amsart}
\usepackage[latin1]{inputenc}
\usepackage{amsmath}
\usepackage{amssymb}
\usepackage{mathrsfs}
\usepackage{amsthm}
\usepackage{graphicx}
\usepackage[pdftex]{color}
\usepackage{paralist}
\usepackage[shortlabels]{enumitem}
\usepackage{adjustbox}
\usepackage[hyperpageref]{backref}
\usepackage[pagebackref]{hyperref}
\renewcommand*{\backref}[1]{}  
   \renewcommand*{\backrefalt}[4]{
      \ifcase #1 
         Not cited.
      \or
         Cited on page #2.
      \else
         Cited on pages #2.
      \fi}
\usepackage{comment}
\usepackage[arrow, matrix, curve]{xy}

\usepackage{tikz} \usetikzlibrary{matrix,patterns,shapes,decorations.pathmorphing,decorations.pathreplacing,calc,arrows,arrows.meta, positioning} 
\usepackage{tikz-cd} 
\usepackage{tkz-graph}

\newcommand{\old}[1]{{\color{red}#1}}

\newtheorem*{corollary*}{Corollary}
\newtheorem{theorem}{Theorem}[section]

\newtheorem{corollary}[theorem]{Corollary}
\newtheorem{lemma}[theorem]{Lemma}
\newtheorem{proposition}[theorem]{Proposition}
\newtheorem{observation}[theorem]{Observation}

\newtheorem{problem}[theorem]{Problem}

\newtheorem*{claim*}{Claim}

\theoremstyle{definition}
\newtheorem{definition}[theorem]{Definition}

\newtheorem{main conjecture}[theorem]{Main Conjecture}
\newtheorem*{theorem }{Theorem}
\newtheorem{remark}[theorem]{Remark}
\newtheorem{conjecture}[theorem]{Conjecture}
\newtheorem{example}[theorem]{Example}

\theoremstyle{remark}

\numberwithin{equation}{section}

\makeatletter
\renewcommand*\env@matrix[1][\arraystretch]{%
  \edef\arraystretch{#1}%
  \hskip -\arraycolsep
  \let\@ifnextchar\new@ifnextchar
  \array{*\c@MaxMatrixCols c}}
\makeatother

\renewcommand{\mod}{\operatorname{mod}}
\newcommand{\proj}{\operatorname{proj}}
\newcommand{\inj}{\operatorname{inj}}

\newcommand{\Ext}{\operatorname{Ext}}

\newcommand{\End}{\operatorname{End}}

\newcommand{\fidim}{\operatorname{fidim}}

\newcommand{\fpdim}{\operatorname{fpdim}}

\newcommand{\Hom}{\operatorname{Hom}}
\newcommand{\add}{\operatorname{\mathrm{add}}}

\newcommand{\rad}{\operatorname{\mathrm{rad}}}

\renewcommand{\mod}{\operatorname{mod}}

\newcommand{\Z}{\mathbb{Z}}

\newcommand{\RHom}{\operatorname{\mathsf{R}Hom}}
\newcommand{\Lotimes}{\otimes^\mathsf{L}}

\newcommand{\CC}{\mathcal{C}}
\renewcommand{\AA}{\mathcal{A}}

\newcommand{\Cok}{\operatorname{Cok}}

\newcommand{\Image}{\operatorname{Im}}
\newcommand{\Db}{\mathrm{D}^{\mathrm{b}}}
\newcommand{\Kb}{\mathrm{K}^{\mathrm{b}}}
\newcommand{\C}{\mathrm{C}}
\newcommand{\Cb}{\mathrm{C}^{\mathrm{b}}}

\newcommand{\domdim}{\operatorname{domdim}}
\newcommand{\codomdim}{\operatorname{codomdim}}
\newcommand{\idim}{\operatorname{idim}}

\newcommand{\pdim}{\operatorname{pdim}}
\newcommand{\gldim}{\operatorname{gldim}}

\begin{document}
\title[Fractionally Calabi--Yau algebras and cluster tilting]{Fractionally Calabi--Yau algebras and cluster tilting}

\date{\today}

\subjclass[2010]{16G10, 16E65}

\keywords{}

\author[Chan]{Aaron Chan}
\address[Chan]{Graduate School of Mathematics, Nagoya University, Furocho, Chikusaku, Nagoya 464-8602, Japan}
\email{aaron.kychan@gmail.com}

\author[Iyama]{Osamu Iyama}
\address[Iyama]{Graduate School of Mathematical Sciences, University of Tokyo, 3-8-1 Komaba Meguro-ku Tokyo 153-8914, Japan}
\email{iyama@ms.u-tokyo.ac.jp}

\author[Marczinzik]{Ren\'{e} Marczinzik}
\address[Marczinzik]{Mathematical Institute of the University of Bonn, Endenicher Allee 60, 53115 Bonn, Germany}
\email{marczire@math.uni-bonn.de}

\dedicatory{}

\begin{abstract}
We show that the class of twisted fractionally Calabi--Yau algebras of finite global dimension coincides with the stable endomorphism algebras of $d$-cluster tilting modules over $d$-representation-finite algebras. 
This is an application of our main result stating that an algebra $A$ of finite global dimension is twisted fractionally Calabi--Yau if and only if there exists $i$ such that the replicated algebra $A^{(i)}$ is a higher Auslander algebra if and only if there exist infinitely many $i$  such that $A^{(i)}$ is a higher Auslander algebra. This gives a new connection between the study of higher Auslander--Reiten theory and twisted fractionally Calabi--Yau algebras, and provides a new construction of large classes of higher Auslander algebras and higher representation-finite algebras. 
We give several applications such as an explicit characterisation of twisted $\frac{n}{2}$-Calabi--Yau algebras, and a triangle equivalence between the bounded derived category of a twisted fractionally Calabi--Yau algebra of finite global dimension and the $\mathbb{Z}$-graded stable module category of an associated higher preprojective algebra.
\end{abstract}
\maketitle
\tableofcontents

\section{Introduction}

\subsection{Fractionally Calabi--Yau algebras and stable higher Auslander algebras}
A triangulated category $\mathcal{T}$ with Serre functor $\nu$ is called \emph{Calabi--Yau} of Calabi--Yau dimension $m$ if $\nu$ is isomorphic to the shift $[m]$ for some $m>0$. We refer for example to \cite{Ke1} for a survey on Calabi--Yau triangulated categories. Fractionally Calabi--Yau categories were introduced by Kontsevich \cite{Kon} as a generalisation of Calabi--Yau categories as follows: A triangulated category $\mathcal{T}$ with Serre functor $\nu$ is called \emph{fractionally Calabi--Yau} of dimension $(m,\ell)$ if $\nu^\ell$ is isomorphic to the shift $[m]$ for some $\ell >0$.
The concept of fractionally Calabi--Yau categories found several applications in algebraic geometry, representation theory and mathematical physics, we refer for example to \cite{CDIM,CC,FK,HI,HIMO,Ke2,Ke3,KLM1,KLM2,K}.

We focus in this article on finite dimensional $K$-algebras $A$, with $K$ a field. Such an algebra $A$ is called \emph{Iwanaga--Gorenstein} if $\idim A_A = \idim _{A}A<\infty$. This is known to be equivalent to the condition that the perfect derived category $\Kb(\proj A)$ has a Serre functor, see for example \cite[Theorem 6.4.6]{Kr}.

\begin{definition}
A finite dimensional Iwanaga--Gorenstein algebra $A$ over a field $K$ is \emph{$(m,\ell)$-fractionally Calabi--Yau} if $\Kb(\proj A)$ is fractionally Calabi--Yau of dimension $(m,\ell)$. More generally, it is called  \emph{twisted $(m,\ell)$-fractionally Calabi--Yau} if there is an isomorphism between $\nu^\ell$ and $[m]$ up to a twist by an automorphism of $A$ \cite{HI,CDIM}.
\end{definition}

When it cannot lead to confusion, we also follow the usual convention to write $\frac{m}{\ell}$-Calabi--Yau instead of the more precise notion of $(m,\ell)$-fractionally Calabi--Yau.
Of special importance are algebras of finite global dimension, where $\Kb(\proj A)$ coincides with the bounded derived category $\Db(\mod A)$.
Many important classes of finite dimensional algebras are twisted fractionally Calabi--Yau such as Dynkin type path algebras, selfinjective algebras, certain $d$-canonical algebras 
and incidence algebras of certain lattices, we refer to \cite{MY,CDIM,Chap,G,HI,Kle,R,Ye} for those examples and more. 

In this article we give a new relation between twisted fractionally Calabi--Yau algebras and higher Auslander--Reiten theory.
We recall the following notion.

\begin{definition}
Let $A$ and $B$ be finite dimensional algebras, and $d\ge1$.
\begin{enumerate}
\item We call $A$ \emph{$d$-Auslander algebra} if $\gldim A \leq d+1 \leq \domdim A$.
\item We call $M\in\mod B$ \emph{$d$-cluster tilting} if
\begin{align*}
\add M&= \{X \in \mod B \mid \Ext_B^i(X,M)=0 \ \text{for} \ i=1,...,d-1 \}\\
&=\{X \in \mod B \mid \Ext_B^i(M,X)=0 \ \text{for} \ i=1,...,d-1 \}.
\end{align*}
In this case $\End_B(M)$ is called a \emph{$d$-Auslander algebra} of $B$ and $\underline{\End}_B(M)$ is called a \emph{stable $d$-Auslander algebra} of $B$.
\end{enumerate}
\end{definition}

These two definitions of $d$-Auslander algebras are actually equivalent. This is a consequence of the \emph{higher Auslander correspondence} between $d$-Auslander algebras with algebras admitting a $d$-cluster tilting modules \cite{I1}.
For $d=1$, this is the classical \emph{Auslander correspondence} between ($1$-)Auslander algebras and representation-finite algebras
\cite[Chapter VI.5]{ARS}.
Studies of these notions have grown in recent years to their own subfields, and are known to be connected to various areas in representation theory and homological algebra; we refer to, for example, \cite{BMR,BMRRT,IY,I1,I2}.

The following class of algebras are especially important in higher Auslander-Reiten theory:
\begin{definition}
An algebra $B$ is called a \emph{$d$-representation-finite algebra} for $d \geq 1$ if $\gldim B \leq d$\footnote{The condition $\gldim B\leq d$ is not assumed in some references, e.g.\ \cite{DI}.} and there exists a $d$-cluster tilting module $M$ in $\mod B$.
\end{definition}

In this case, $M$ is unique up to additive equivalence, and hence $d$-Auslander algebras of $B$ are unique up to Morita equivalence \cite{I2}.
The $d$-representation-finite algebras generalise the classical hereditary algebras of Dynkin type, which correspond to the case when $d=1$, and include important families such as the higher Auslander algebras of type $A_n$ that appear in various areas in mathematics, see for example \cite{DJL,OT,Wi}.  $d$-representation-finite algebras enjoy many nice properties, for instance, each of them has an associated higher preprojective algebra \cite{IO} which is finite-dimensional and selfinjective, and that $\add M$ is a so-called $d$-abelian category in the sense \cite{J}.


Our main result of this article gives a surprising new homological characterisation of twisted fractionally Calabi-Yau algebras of finite global dimension via stable higher Auslander algebras:

\begin{theorem}\label{theorem in intro} {\rm (Theorem \ref{stableendomainresult})}
For an algebra $A$ of finite global dimension, the following conditions are equivalent.
\begin{enumerate}[\rm(1)]
    \item $A$ is Iwanaga--Gorenstein and twisted fractionally Calabi--Yau.
    \item There exists a finite dimensional $d$-representation-finite algebra $B$ for some $d \geq 1$ such that $A$ is isomorphic to the stable $d$-Auslander algebra of $B$.    
\end{enumerate}
\end{theorem}

Theorem \ref{theorem in intro} gives a homological characterization of stable higher Auslander algebras as twisted fractionally Calabi-Yau algebras.
More strongly, it follows from Theorem \ref{stableendomainresult} 
that each twisted fractionally Calabi--Yau algebra can be realized as a stable higher Auslander algebra in infinitely many different ways.

The study of those stable Auslander algebras appears in many places in representation theory, most prominently in the fundamental work by Auslander and Reiten in \cite{AR1,AR2,AR3,AR4,AR5,Re}.
The implication (2)$\Rightarrow$(1) 
was already known by \cite[Theorem 8.7]{CDIM}. 
Thus the contribution of this paper is the converse implication (1)$\Rightarrow$(2).

\subsection{Fractionally Calabi--Yau algebras and minimal Auslander--Gorenstein algebras}
We establish the converse implication (1)$\Rightarrow$(2) by  giving a new characterisation of general twisted fractionally Calabi--Yau algebras, not necessarily of finite global dimension.
This new characterisation is found using the so-called replicated algebras.
For this, we recall some more concepts in the following.
In \cite{IyaSol}, Iyama and Solberg investigated a generalisation of higher Auslander correspondence, where cluster tilting modules are replaced by precluster tilting modules, and higher Auslander algebras are replaced by \emph{minimal Auslander--Gorenstein algebras}.  Such an algebra $A$ is an Iwanaga--Gorenstein algebra that satisfies $\idim A \leq d \leq \domdim A$.
Replicated algebras were investigated in \cite{AI} as a variant of trivial extensions algebras that have positive dominant dimension.
The $m$-th replicated algebra of a finite dimensional algebra $A$ is defined as a $(m+1)\times (m+1)$ matrix algebra:
\[A^{(m)}=\left(\begin{array}{ccccc}
A & 0 & \cdots & 0 & 0  \\
DA & A & \cdots & 0 & 0 \\
\vdots & \ddots & \ddots & \vdots & \vdots \\
0 & 0 & \ddots & A  & 0 \\
0 & 0 & \cdots & DA & A
\end{array}\right).
\]
Here $DA=\Hom_K(A,K)$ is the dual of the regular module, considered as an $A$-bimodule with trivial multiplication.
Replicated algebras are closely related to repetitive algebras \cite{H} and they found  applications in the study of $m$-cluster categories, derived algebraic geometry and related fields see for example \cite{ABST,W,X}. In the most fundamental case when $m=1$, the algebra $A^{(m)}$ is usually simply called the \emph{duplicated} algebra of $A$, we refer to the articles \cite{ABST2,BS} for more information. We call an algebra $B$ a \emph{replicated algebra} when $B$ is isomorphic to $A^{(m)}$ for some $m \geq 1$ for some finite dimensional algebra $A$.

The study of the homological dimensions of replicated algebras and related matrix algebras is a classical topic with several partial results \cite{AI,CIM,PR}. In this article we will give a complete answer when the selfinjective, global and dominant dimensions of replicated algebras coincide and use this to prove our first main result (see Corollary \ref{mainresultsec3 cor} for details).

\begin{theorem} \label{mainresult} {\rm (Corollary \ref{mainresultsec3 cor})}
Let $A$ be a finite-dimensional algebra over a field $k$.
Then the following are equivalent:
\begin{enumerate}[\rm(1)]
    \item $A$ is twisted fractionally Calabi--Yau.
    \item There exists an $i>0$ such that $A^{(i)}$ is a minimal Auslander--Gorenstein algebra. 
    \item There exist infinitely many $i>0$ such that $A^{(i)}$ is a minimal Auslander--Gorenstein algebra.
\end{enumerate}
Furthermore, if we restrict to algebras $A$ with $\gldim A < \infty$ in (1), then we can replace ``minimal Auslander--Gorenstein" by ``higher Auslander algebra" in conditions (2) and (3).
\end{theorem}

This result vastly generalises the main result of \cite{CIM} on Serre-formal algebras, see Theorem \ref{Serre-formal and stalk} for the explicit relationship. There is a strong connection between Theorem \ref{mainresult} and the remarkable results of Sen \cite{Se}.
We remark that there are other characterisations of fractionally Calabi--Yau algebras of finite global dimension, for example, \cite{CDIM} does this via the periodicity of their trivial extension algebras.

Since Theorem \ref{mainresult} says that a single twisted fractionally Calabi--Yau algebra of finite global dimension generates an infinite family of higher Auslander algebras; this provides an extensive -- perhaps the most extensive known so far -- construction of higher Auslander algebras and cluster tilting modules.  
This opens up other practical and theoretical applications.  For example, determining whether an algebra is fractionally Calabi--Yau via computer algebra systems like \cite{QPA} can be used to test various conjectures, which we will do in the future.

We give two other applications of our Theorem \ref{mainresult} in this paper.
The first application gives a description of the derived category and cluster category of twisted fractionally Calabi--Yau algebras using the stable and derived categories of higher preprojective algebras by combining our main result with results from \cite[Section 4]{IO}:

\begin{corollary}{\rm (Corollary \ref{corollaryhigherpreprojective})}
Let $A$ be a twisted $\frac{m}{\ell}$-Calabi--Yau algebra of finite global dimension, where $\ell\ge2$ and $m\ge0$.
Then $A^{(\ell-2)}$ is $n$-representation-finite for $n:=\ell+m-2$. Moreover, let $\Pi:=\Pi_{n+1}(A^{(\ell-2)})$ denote the $(n+1)$-preprojective algebra of $A^{(\ell-2)}$.
Then we have triangle equivalences
\[\underline{\mod}^{\mathbb{Z}} \Pi \cong \Db(\mod A)\ \mbox{ and }\ \underline{\mod}\,\Pi \cong \mathcal{C}_{n+1}(A),\]
where $\mathcal{C}_{n+1}(A)$ denotes the $(n+1)$-cluster category of $A$. 
 \end{corollary}

Notice that a finite dimensional algebra is selfinjective if and only if it is twisted $\frac{0}{1}$-Calabi--Yau if and only if it is twisted $\frac{n}{1}$-Calabi--Yau for some $n\in\Z$. This raises naturally the question of a homological characterisation of twisted $\frac{n}{2}$-Calabi--Yau algebras with $n\in\Z$. As our next application of our Theorem \ref{mainresult}, we give the following equivalent conditions.

\begin{theorem}{\rm (Theorem \ref{onehalfCYcor})}
The following are equivalent for an algebra $A$:
\begin{enumerate}[\rm(1)]
    \item $A$ is twisted $\frac{n}{2}$-Calabi--Yau.
    \item $A^{(1)}$ is a minimal Auslander--Gorenstein algebra of selfinjective dimension $n+1$.
    \item $A$ is Iwanaga--Gorenstein with $\idim A=n$, $\Hom_A(DA,A)=0$ and $A \oplus DA$ is an $n$-precluster tilting object.
\end{enumerate}
\end{theorem}

For any given $n$, it is an interesting question to classify algebras satisfying these conditions.
When $A$ is of infinite global dimension, this problem is open even in the case $n=1$; in the finite global dimension case with $n=1$, such a classification was given in \cite[Remark 1.4]{AI}.


In Section 2 we will give the necessary preliminaries and in Section 3 we prove Theorem \ref {mainresult}, which we use in Section 4 to prove the remaining main results of this article.

\medskip\noindent
{\bf Acknowledgements. }
This research is motivated by experiments with the computer algebra package \cite{QPA}. AC is supported by JSPS Grant-in-Aid for Scientific Research (C) 24K06666.  OI is supported by JSPS Grant-in-Aid for Scientific Research (B) 22H01113 and (B) 23K22384.
RM was supported by the DFG with the project number 428999796.

\section{Preliminaries on replicated algebras}

In this section we fix our conventions and review the construction of replicated algebras and the description of their modules.
Algebras are always finite dimensional over a field $K$ and modules are finitely generated right modules, unless otherwise stated. $D=\Hom_K(-,K)$ denotes the $K$-linear duality. We refer for example to \cite{ARS} and \cite{Kr} for an introduction to representation theory and homological algebra of finite dimensional algebras.
We refer for example to \cite{AI} and \cite{CIM} for more information and examples of replicated algebras.
We will follow the exposition in \cite{CIM} in this section very closely.
Let 
$$0 \rightarrow A_A\rightarrow I^0 \rightarrow I^1 \rightarrow \cdots $$
be a minimal injective coresolution of the regular module $A$. Then the \emph{selfinjective dimension} of $A$ is the maximal $n$ such that $I^n \neq 0$; or infinite if $I^n \neq 0$ for all $n\geq 0$.  The \emph{dominant dimension} of $A$ is defined as the smallest $n \geq 0$ such that $I^n$ is not projective or as infinite if no such $n$ exists.

For a finite dimensional algebra $A$, there is an infinite dimensional locally bounded algebra called \emph{repetitive algebra} $\widehat{A}$ defined as follows:
The underlying vector space is given by $(\bigoplus_{i\in\Z}A)\oplus(\bigoplus_{i\in\Z}DA)$, with elements of the form $(a_i,f_i)_{i\in\Z}$ with finitely many non-zero $a_i$ and $f_i$'s.
The multiplication is defined by $(a_i,f_i)_i(b_i,g_i)_i=(a_ib_i,a_{i+1}g_i+f_ib_i)_i$.
Then $\widehat{A}$ can be regarded as a $\Z\times \Z$ matrix (non-unital but locally unital) algebra with diagonal entries $\widehat{A}_{i,i}=A$, sub-diagonal entries $\widehat{A}_{i,i-1}=DA$, and everywhere else zero.

\begin{definition}
Let $A$ be a finite dimensional algebra.
Denote by $e^{(m+1)}$ the idempotent of $\widehat{A}$ given by the $\Z\times\Z$ matrix with $(i,i)$-th entry $1$ for all $i\in\{1,\ldots,m+1\}$ and zero everywhere else.
For $m\geq 0$, the \emph{$m$-replicated algebra} $A^{(m)}$ of $A$ is the idempotent truncation $e^{(m+1)}\widehat{A}e^{(m+1)}$.
In other words, we have
\begin{align}
A^{(m)}:=e^{(m+1)} \widehat{A}e^{(m+1)} =
\left(\begin{array}{cccccc}
A & 0 & 0 & \cdots & 0 & 0 \\
DA & A & 0 &  & & 0 \\
0 & DA & A &  & & 0 \\
\vdots & \ddots & \ddots & \ddots & &\vdots \\
0 & & \ddots & \ddots & \ddots &\vdots \\
0 & 0 & 0 & \cdots & DA & A
\end{array}\right), \notag 
\end{align}
where the matrix is of size $(m+1)\times (m+1)$.

\end{definition}
For many explicit examples of replicated algebras by quiver and relations and a relation to SGC-extension algebras, we refer to \cite{CIM}. We remark that explicit quiver and relations of repetitive algebras and thus also replicated algebras of quiver algebras containing only monomial or commutativity relations can be obtained from the description in \cite{S}.

Any $A^{(m)}$-module can be described by a sequence
\[
(M_0, f_1, M_1, f_2, \ldots, f_m, M_m) ,
\]
where $M_i\in \mod A$ and $f_i\in \Hom_A(M_i\otimes_ADA,M_{i-1})$ satisfying $f_{i-1}\circ (f_i\otimes 1)=0$ for all $i$.
We denote such a module simply by $(M_i,f_i)_i$ whenever $m$ is understood.
Let $(a_i,\phi_i)_i$ be an element of $A^{(m)}$, and $(x_i)_i$ be an element of the module given by the above sequence.
Then the $A^{(m)}$-action is given by $(x_i)_i\cdot (a_i,\phi_i)_i = (x_ia_i+ f_{i+1}(x_{i+1}\otimes \phi_{i+1}))_i$.
Using this description, an $A^{(m)}$-module homomorphism can be described by a sequence $(\theta_i)_i: (M_i,f_i)_i \to (N_i,g_i)_i$ with $\theta_i\in \Hom_A(M_i,N_i)$ satisfying the natural commutation relation (c.f. \cite[III.2]{ARS}).

\begin{observation}\label{eg-modules}
Let $M$ be an $A$-module
\begin{enumerate}[\rm(1)]
\item For $1\le k\le m+1$, $[M]_k$ is the ``stalk module" given by a row vector with $m+1$ entries, where the $k$-th entry being $M$ and the remaining zero:
\[
[M]_k := (\underbrace{\stackrel{1}{0}, \ldots, \stackrel{k-1}{0}, \stackrel{k}{M}, \stackrel{k+1}{0}, \ldots, \stackrel{m+1}{0}}_{m+1\text{ entries}});
\]
the action of $A^{(m)}$ is obvious.
Note that, for every $1\le k\le m+1$, $[?]_k:\mod A\to \mod A^{(m)}$ is a full exact embedding of categories which preserves indecomposable modules, almost split sequences, and irreducible morphisms (c.f. \cite[Lemma 5]{ABST}).

\item If $M$ is projective (resp. injective), for $1\le k\le m$, $[\nu_A(M),M]_k$ (resp. $[M,\nu_A^{-1}(M)]_k$) is the $A^{(m)}$-module given by $M_{k+1}=\nu_A(M)$ (resp. $M$), $f_{k+1}=\mathrm{id}$, $M_{k}=M$ (resp. $\nu_A^-(M)$), i.e.
\begin{align}
[\nu_A(M),M]_k &:= (\underbrace{\stackrel{1}{0}, \ldots, \stackrel{k-1}{0}, \stackrel{k}{\nu_A(M)}, \stackrel{k+1}{M}, \stackrel{k+2}{0},\ldots, \stackrel{m+1}{0}}_{m+1\text{ entries}}) \quad \text{for projective $M$;}\notag \\
[M,\nu_A^-(M)]_k &:= (\underbrace{\stackrel{1}{0}, \ldots, \stackrel{k-1}{0}, \stackrel{k}{M}, \stackrel{k+1}{\nu_A^-(M)}, \stackrel{k+2}{0}, \ldots, \stackrel{m+1}{0}}_{m+1\text{ entries}}) \quad \text{for injective $M$.}\notag 
\end{align}
This also shows that $[\nu_A(M),M]_k$ is projective-injective for an indecomposable projective $A$-module $M$ as it is a direct summand of the $k$-th row of $A^{(m)}$, as well as the dual of a direct summand of the $k$-th column of $A^{(m)}$.

\item The injective envelope $I_M$ and cosyzygy $\Omega_A^-(M)$ form a short exact sequence
\[
0\to M\xrightarrow{\iota} I_M \xrightarrow{\pi} \Omega_A^-(M)\to 0.
\]
For $1\le k\le m$, the injective envelope of $M$ is $[I_M,\nu_A^-(I_M)]_k$.
Moreover, we have a short exact sequence
\begin{align}\label{eq-cosyzygy}
0\to [M]_k\xrightarrow{[\iota,0]_k} [I_M,\nu_A^-(I_M)]_k \xrightarrow{[\pi,\mathrm{id}]_k} \Omega_{A^{(m)}}^-([M]_k)\to 0,
\end{align}
where the cosyzygy $\Omega_{A^{(m)}}^-([M]_k)$ of $[M]_k$ is given by $(M_i,f_i)_i$ with $M_{k+1}=\Omega_A^-(M)$, $f_{k+1}=\pi$, $M_{k}=\nu_A^-(I_M)$, and everything else zero.
In particular, if $M$ is injective, then $\Omega_{A^{(m)}}^-([I]_k)=[\nu_A^-(I)]_{k+1}$.
Note that $[\pi,\mathrm{id}]_k$ is the map given by the sequence $(\theta_i)_{i=0,1,\ldots, m}$ with $\theta_k=\pi, \theta_{k+1}=\mathrm{id}, \theta_i=0$ for all $i\neq k$; similarly for $[\iota,0]_k$.
We leave it to the reader to write down the analogous description of $\Omega_{A^{(m)}}([M]_k)$.
\end{enumerate}
\end{observation}

In addition to (2), we remark that
\begin{align}\label{proj inj}
\proj A^{(m)} =& \add\{ [A]_1, [DA,A]_i \mid i=1,\ldots,m \}; \notag \\
\inj A^{(m)} =& \add\{ [DA]_{m+1}, [DA,A]_i \mid i=1,\ldots,m \}; \notag \\
\proj A^{(m)}\cap \inj A^{(m)} =& \add\{ [DA,A]_i \mid i=1,\ldots,m \}.
\end{align}

Recall that the \emph{finitistic projective dimension} $\fpdim A$ of a finite dimensional algebra $A$ is defined as the supremum of all finite projective dimensions of $A$-modules.
Dually, one defines the finitistic injective dimension $\fidim A$ of an algebra $A$. When $A$ is Iwanaga--Gorenstein with $\idim A=g$, then
$g=\fpdim A=\fidim A$, see for example Corollary 3.2.6 and Theorem 3.2.7 in \cite{Che}.
We will need the following result:

\begin{proposition}\label{IwanagaGortriangularmatrixring}
For rings $R$, $S$ and $S$-$R$-bimodule $M$, let $\Lambda= \begin{pmatrix}
R&0 \\
M&S
\end{pmatrix}$. 
\begin{enumerate}[\rm(1)]
\item The following conditions are equivalent:
\begin{enumerate}[\rm(1)]
    \item $R$ and $S$ are Iwanaga--Gorenstein and $\pdim M_R < \infty$ and $\pdim {}_{S}M < \infty$.
    \item $\Lambda$ is Iwanaga--Gorenstein.
\end{enumerate}
\item If $M\neq0$ and $\pdim M_R < \infty$, then 
$$ \pdim M_R +1 \leq \fpdim \Lambda.$$
\end{enumerate}
\end{proposition}

\begin{proof}
(1) See, e.g.\ \cite[Lemma 2.1, 2.3]{XZ}.

(2) The $\Lambda$-module $Y:=[0\ S]$ has first syzygy $\Omega_{\Lambda}^1(Y)=[M\ 0]$ and a minimal projective resolution of $M$ as an $R$-module lifts to a minimal projective resolution as a $\Lambda$-module of $[M\ 0]$. Thus $\pdim_{\Lambda} Y=\pdim_{\Lambda} \Omega_{\Lambda}^1(Y)+1=\pdim_R M +1$. In particular, $\pdim M_R +1 \leq \fpdim \Lambda$.
\end{proof}

\begin{remark} \label{remarkFGR}
There is also an upper bound $\fpdim\Lambda \leq \max(\fpdim R+\pdim {}_{R}M+1,\fpdim S)$, see \cite[Corollary 4.21]{FGR}.
\end{remark}


We will also need the following proposition:

\begin{proposition}\label{Iwanaga--Gorenstein and replicated}
The following are equivalent:
\begin{enumerate}[\rm(1)]
    \item $A$ is Iwanaga--Gorenstein.
    \item $A^{(m)}$ is Iwanaga--Gorenstein for all $m\geq 0$.
    \item $A^{(m)}$ is Iwanaga--Gorenstein for some $m \geq 0$.
\end{enumerate}

If $A$ is Iwanaga--Gorenstein, then
$$\idim A_A\leq\idim A^{(1)}-1 \leq\idim A^{(2)}-2\leq\cdots \leq \idim A^{(m)}-m.$$
\end{proposition}

\begin{proof}
That (2) implies (3) is clear. Now assume (3), then by Proposition \ref{IwanagaGortriangularmatrixring}(1), it follows that $A$ is Iwanaga--Gorenstein using 
\begin{equation}\label{Am induction}
A^{(m)}=\left(
\begin{array}{cc}
A^{(m-1)} & 0 \\
\ [DA]_{m} & A
\end{array}
\right).
\end{equation}
and thus (1) holds. 
Now assume (1), then (2) follows by inductively using the formulation \eqref{Am induction} and applying Proposition \ref{IwanagaGortriangularmatrixring}(1) with $R= A^{(m-1)}$, $S=A$ and $M=[DA]_m$ (since $\pdim_A[DA]_m< \infty $ and $\pdim([DA]_m)_{A^{(m-1)}}< \infty$ hold). 

For the final claim, now assume that $A$ (and thus also $A^{(m)}$ for all $m \geq 0$) is Iwanaga--Gorenstein. Then we have  $\idim A^{(m)} = \fpdim A^{(m)}$.
As in the proof of (1)$\Rightarrow$(2), the formulation \eqref{Am induction} allows us to apply Proposition \ref{IwanagaGortriangularmatrixring}(2).  Since 
$$\pdim [D A]_m = \idim A^{(m-1)}$$
as an $A^{(m-1)}$-module, we get $\idim A^{(m-1)}+1 \leq \idim A^{(m)}$, as desired.
\end{proof}
\begin{remark}
With the inequality from the previous Remark \ref{remarkFGR}, one can also show the upper bound $\idim A^{(m)} \leq (m+1) \idim A+m$ using induction.
\end{remark}



\section{Replicated algebras of fractionally Calabi--Yau algebras}

\subsection{Relation with minimal Auslander--Gorenstein algebras}
In this section we give a proof of Theorem \ref{mainresult}. Note that the twisted $\frac{m}{\ell}$-Calabi--Yau algebras with $\ell=1$ are exactly the selfinjective algebras and thus we restrict our attention to $\ell \geq 2$.
We prove the following more explicit results.

\begin{theorem} \label{mainresultsec3}
Let $A$ be a finite dimensional algebra over a field $k$, and $\ell\ge2$ and $m\ge0$ be integers. Then the following are equivalent:
\begin{enumerate}[\rm(1)]
    \item $A$ is Iwanaga--Gorenstein and twisted $(m,\ell)$-Calabi--Yau.
    \item $A^{(\ell-1)}$ is a minimal Auslander--Gorenstein algebra of selfinjective dimension $\ell+m-1$.
    \end{enumerate}
In particular, the following are equivalent:
\begin{enumerate}[\rm(1)]
    \item $A$ has finite global dimension and is twisted $(m,\ell)$-Calabi--Yau.
    \item $A^{(\ell-1)}$ is a higher Auslander algebra of global dimension $\ell+m-1$.
    \end{enumerate}
\end{theorem}

We start with the following basic observation.

\begin{lemma}\label{total complex}
Let $\AA$ be an abelian category. 
Suppose that we have morphisms $a_i:Y_i\to X_i$ in $\C(\AA)$ for $1\le i\le n$ and let $b_i:X_i\to{\rm cone}(a_i)$ be the induced morphisms.
If there are quasi-isomorphisms $c_i:{\rm cone}(a_i)\to Y_{i-1}$ in $\C(\AA)$ for $1\le i\le n$,
then the total complex of
\begin{equation*}\label{n}
Y_n\xrightarrow{a_n} X_n\xrightarrow{a_{n-1}c_nb_n}\cdots\xrightarrow{a_1c_2b_2}X_1\xrightarrow{c_1b_1}Y_0
\end{equation*}
is acyclic.
\end{lemma}

\begin{proof}
Let $T$ be the desired total complex. Then $T$ coincides with 
the total complex of ${\rm cone}(a_1)\xrightarrow{c_1}Y_0$ if $n=1$, and the total complex of
\[{\rm cone}(a_n)\xrightarrow{a_{n-1}c_n} X_{n-1}\xrightarrow{a_{n-2}c_{n-1}b_{n-1}}\cdots\xrightarrow{a_1c_2b_2}X_1\xrightarrow{c_1b_1}Y_0\]
if $n\ge2$. Since $c_1$ is a quasi-isomorphism, $T$ is acyclic in the former case.
In the latter case, since $c_n:{\rm cone}(a_n)\to Y_{n-1}$ is a quasi-isomorphism, $T$ is quasi-isomorphic to the total complex of
\[Y_{n-1}\xrightarrow{a_{n-1}} X_{n-1}\xrightarrow{a_{n-2}c_{n-1}b_{n-1}}\cdots\xrightarrow{a_1c_2b_2}X_1\xrightarrow{c_1b_1}Y_0.\]
Thus the assertion follows inductively.
\end{proof}

Now we give a general construction.
Let $A$ be a finite dimensional algebra which is Iwanaga--Gorenstein.
Then the isomorphism closures of $\Kb(\proj A)$ and $\Kb(\inj A)$ in $\Db(\mod A)$ coincide, and the Nakayama functor $\nu$ can be regarded as an equivalence
\[\nu:\Kb(\proj A)\simeq\Kb(\proj A).\]
For $i\ge0$, define $Q_i\in\Cb(\proj A)$ by taking projective resolutions
\[f_0:Q_0\to DA\ \text{ and }\ f_i:Q_i\to \nu(Q_{i-1})\ \text{ for }\ i\ge1.\]
Then we have
\begin{equation}\label{Q_i}
 Q_i\simeq\nu^i(DA)=\nu^{i+1}(A)\ \text{ in }\ \Db(\mod A)\ \text{ for }\ i\ge0.
\end{equation}
For each $1\le i\le\ell-1$, we have morphisms 
\[[Q_i]_{\ell-i}\xrightarrow{a_i:=[f_i,0]_{\ell-i}}[\nu(Q_{i-1}),Q_{i-1}]_{\ell-i}\xrightarrow{b_i}{\rm cone}(a_i)=[{\rm cone}(f_i),Q_{i-1}]_{\ell-i}\]
and a quasi-isomorphism
\[c_i:=\left\{\begin{array}{rll}[0,1]_{\ell-i}:&{\rm cone}(a_i)=[{\rm cone}(f_i),Q_{i-1}]_{\ell-i}\to[Q_{i-1}]_{\ell-i+1}&\text{if }i\ge2\\
\phantom{.} [0,f_0]_{\ell-1}:&{\rm cone}(a_1)=[{\rm cone}(f_1),Q_0]_{\ell-1}\to[DA]_{\ell}&\text{if }i=1\end{array}\right.\]
in $\Cb(\proj A^{(\ell-1)})$.
It follows from Lemma \ref{total complex} that the total complex of
\begin{equation}\label{key sequence}
[Q_{\ell-1}]_1\xrightarrow{(f_{\ell-1})_1}[\nu(Q_{\ell-2}),Q_{\ell-2}]_1\xrightarrow{(f_{\ell-2})_2}\cdots\xrightarrow{(f_2)_{\ell-2}}[\nu(Q_1),Q_1]_{\ell-2}\xrightarrow{(f_1)_{\ell-1}}[\nu(Q_0),Q_0]_{\ell-1}\xrightarrow{(f_0)_\ell}[DA]_\ell.
\end{equation}
is exact.
We denote the total complex of \eqref{key sequence} by 
\begin{equation}\label{key sequence 2}
   \cdots\to X^{-2}\to X^{-1}\to [DA]_\ell\to0.
\end{equation}
with $X^{-i}\in \mod A$ for $i\ge 1$.
Since each term of the double complex \eqref{key sequence} is projective as an $A^{(\ell-1)}$-module except for $[DA]_\ell$, the sequence \eqref{key sequence 2} is a projective resolution of the $A^{(\ell-1)}$-module $[DA]_\ell$. 

We can now show the implication (1) $\Rightarrow$ (2) of Theorem \ref{mainresultsec3}.

\begin{proposition}\label{CY=>AG}
Let $A$ be a finite dimensional algebra, $\ell\ge2$ and $m\ge0$. If $A$ is Iwanaga--Gorenstein and twisted $(m,\ell)$-Calabi--Yau, then $A^{(\ell-1)}$ is minimal Auslander--Gorenstein algebra with selfinjective dimension $\ell+m-1$.
\end{proposition}


\begin{proof}
Recall that all indecomposable injective $A^{(\ell-1)}$-modules are projective except for the direct summands of $[DA]_\ell$, so it suffices to show that $\pdim [DA]_\ell=\ell+m-1=\codomdim[DA]_\ell$.

Since $A$ is twisted $(m,\ell)$-Calabi--Yau, we have isomorphisms in $\Db(\mod A)$:
\[Q_{\ell-1}\stackrel{\eqref{Q_i}}{\simeq}\nu^\ell(A)\simeq A[m].\]
Thus the leftmost term of \eqref{key sequence} can be taken as $[Q_{\ell-1}]_1=[A[m]]_1$. Then the term $X^{-\ell-m}$ in \eqref{key sequence 2} has a direct summand $[A]_1$. Take a decomposition $X^{-\ell-m}=[A]_1\oplus Y$ as an $A^{(\ell-1)}$-module.
Since each column (the $Q_i$'s and $\nu Q_i$'s) of the double complex \eqref{key sequence} is bounded, the total complex \eqref{key sequence 2} is also bounded, say, by $-N\ll 0$. 
That is, the exact sequence \eqref{key sequence 2} is of the form
\[0\to X^{-N}\xrightarrow{d^{-N}} X^{1-N}\xrightarrow{d^{1-N}}\cdots\to X^{-1-\ell-m}\to [A]_1\oplus Y\to X^{1-\ell-m}\to\cdots\to X^{-2}\to X^{-1}\to[DA]_\ell\to0.\]
Since each term of the double complex \eqref{key sequence} is injective as an $A^{(\ell-1)}$-module except for $[A[m]]_1$, so does $Y$ and each $X^i$ with $i\neq-\ell-m$. 
This implies that $d^{-N}:X^{-N}\to X^{1-N}$ is a split monomorphism, and that the morphism $\Cok d^{-N}\to X^{2-N}$ induced by $d^{1-N}$ is a split monomorphism, and so on. Repeating the same argument, in the homotopy category $\Kb(\proj A^{(\ell-1)})$, the sequence \eqref{key sequence 2} is isomorphic to a sequence of the form
\[0\to[A]_1\to Z\to X^{2-\ell-m}\to\cdots\to X^{-2}\to X^{-1}\to[DA]_\ell\to0,\]
where $Z$ is a direct summand of $X^{1-\ell-m}$.
This sequence shows the claim $\pdim [DA]_\ell=\ell+m-1=\codomdim[DA]_\ell$.
\end{proof}

Now we prove the implication (2) $\Rightarrow$ (1) of Theorem \ref{mainresultsec3}.

\begin{proposition}\label{AG=>CY}
Let $A$ be a finite dimensional algebra and $\ell\ge2$ and $n\ge0$. If $A^{(\ell-1)}$ is minimal Auslander--Gorenstein of selfinjective dimension $n$. Then $A$ is Iwanaga--Gorenstein and twisted $(n-\ell+1,\ell)$-Calabi--Yau.
\end{proposition}

\begin{proof}
Since $A^{(\ell-1)}$ is minimal Auslander--Gorenstein, $A$ is Iwanaga--Gorenstein by Proposition \ref{Iwanaga--Gorenstein and replicated},
and the $A^{(\ell-1)}$-module $[DA]_\ell$ has a minimal projective resolution of the form
\begin{equation}\label{key sequence 3}
0\to [A]_1\to Y^{-n}\to\cdots\to Y^{-2}\to Y^{-1}\to [DA]_\ell\to0
\end{equation}
for some $n$, where each $Y^i$ is projective-injective as an $A^{(\ell-1)}$-module.
Moreover, in the category $\Cb(\proj A^{(\ell-1)})$, the complex \eqref{key sequence 2} is a direct sum of the complex \eqref{key sequence 3} and some contractible complex.

Without loss of generality, we can assume that each differential of $Q_{\ell-1}$ is in the radical of $\mod A$. Then, in the double complex \eqref{key sequence}, both horizontal and vertical morphisms starting at each term of $[Q_{\ell-1}]_1$ is in the radical of $\mod A^{(\ell-1)}$.
Therefore each term of the complex $[Q_{\ell-1}]_1$ still appear in the minimal projective resolution \eqref{key sequence 3}.

Since each indecomposable direct summand of each term of $[Q_{\ell-1}]_1$ is non-injective, 
the complex $Q_{\ell-1}$ has only one non-zero term corresponding to the left-most term $[A]_1$ of \eqref{key sequence 3}, that is, $Q_{\ell-1}\simeq A[n-\ell+1]$ in $\Cb(\mod A)$. Now we have
\[\nu^{\ell}(A)\stackrel{\eqref{Q_i}}{\simeq} Q_{\ell-1}\simeq A[n-\ell+1]\]
and hence $A$ is twisted $(n-\ell+1,\ell)$-Calabi--Yau.
\end{proof}

\begin{proof}[Proof of Theorem \ref{mainresultsec3}]
The first part follows from Propositions \ref{CY=>AG} and \ref{AG=>CY}.

Now the last part about the finite global dimension follows from \cite[Proposition 1.1]{AI}, where it was shown that an algebra $A$ has finite global dimension if and only if $A^{(m)}$ has finite global dimension for some (also, for all) $m>0$.
\end{proof}


Immediately, we obtain the following observation.

\begin{corollary} \label{mainresultsec3 cor}
Let $A$ be a finite-dimensional algebra over a field $k$.
Then the following are equivalent:
\begin{enumerate}[\rm(1)]
    \item $A$ is Iwanaga--Gorenstein and twisted fractionally Calabi--Yau.
    \item There exists an $i>0$ such that $A^{(i)}$ is a minimal Auslander--Gorenstein algebra. 
    \item There exists $\ell\ge2$ such that for each $i\ge1$, $A^{(\ell i-1)}$ is a minimal Auslander--Gorenstein algebra.
\end{enumerate}
Furthermore, if we restrict to algebras $A$ with $\gldim A < \infty$ in (1), then we can replace ``minimal Auslander--Gorenstein" by ``higher Auslander algebra" in conditions (2) and (3).
\end{corollary}

\begin{proof}
(3)$\Rightarrow$(2) This is clear.

(2)$\Rightarrow$(1) By Proposition \ref{Iwanaga--Gorenstein and replicated}, we have $\idim A^{(i)}\ge i$. Let $\ell:=i+1\ge2$ and $m:=\idim A^{(i)}-i\ge0$. By Theorem \ref{mainresultsec3}(2)$\Rightarrow$(1), $A$ is Iwanaga--Gorenstein and twisted $(m,\ell)$-Calabi--Yau.

(1)$\Rightarrow$(3) This follows from Theorem \ref{mainresultsec3}(1)$\Rightarrow$(2) and the fact that twisted $(m,\ell)$-Calabi--Yau algebras are also twisted $(mi,\ell i)$-Calabi--Yau for each $i\ge1$.
\end{proof}

We can formulate this result as a bijective correspondence between twisted fractionally Calabi--Yau algebras and replicated minimal Auslander--Gorenstein algebras as follows:
\begin{corollary} \label{mainresultascorrespondence}
For each integers $\ell,m$ with $\ell\ge2$ and $m\ge0$, there is a bijection 
\begin{align*}
{\Big\{
    \text{twisted $(m,\ell)$-Calabi--Yau algebras $A$}
\Big\}}_{/\simeq}
\;\longleftrightarrow\;
{\begin{Bmatrix}
    \text{$(\ell-1)$-replicated minimal} \\ \text{Auslander--Gorenstein algebras $B$}\\ \text{with selfinjective dimension $\ell+m-1$}
\end{Bmatrix}}_{/\simeq}
\end{align*}
which restricts to a bijection
\begin{align*}
{\begin{Bmatrix}
    \text{twisted $(m,\ell)$-Calabi--Yau algebras $A$} \\ 
    \text{of finite global dimension}
\end{Bmatrix}}_{/\simeq}
\;\longleftrightarrow\;
{\begin{Bmatrix}
    \text{$(\ell-1)$-replicated} \\ 
    \text{higher Auslander algebras $B$}\\ \text{with global dimension $\ell+m-1$}
\end{Bmatrix}}_{/\simeq}.
\end{align*}
The correspondence sends a twisted Calabi--Yau algebra $A$ to $A^{(\ell-1)}$, and a replicated minimal Auslander--Gorenstein algebra $B$ to the algebra $eBe$, where $e$ denotes the idempotent corresponding to the term $A$ in the upper-left $(1,1)$-corner of the replicated algebra.
\end{corollary}

\begin{proof}
The map $A\mapsto B$ is well-defined by Theorem \ref{mainresultsec3}(1)$\Rightarrow$(2).
Also the map $B\mapsto A$ is well-defined by Theorem \ref{mainresultsec3}(2)$\Rightarrow$(1) since $B=A^{(\ell-1)}$ holds for $A:=eBe$.
These maps are clearly inverse to each other.
The last part of Theorem \ref{mainresultsec3} shows that the correspondence restricts to algebras of finite global dimension.
\end{proof}


\subsection{Relation with stable higher Auslander algebras}

We start with giving additional information on Theorem \ref{mainresultsec3}.

\begin{proposition}\label{mainresultsec3+}
Let $A$ be a finite dimensional algebra which has finite global dimension and is twisted $(m,\ell)$-Calabi--Yau with $\ell\ge2$ and $m\ge0$. If $n:=\ell+m-2\ge1$, then $A^{(\ell-2)}$ is an $n$-representation-finite algebra with $n$-cluster tilting module $A^{(\ell-2)}\oplus DA^{(\ell-2)}$ whose $n$-Auslander algebra is $A^{(\ell-1)}$ and stable $n$-Auslander algebra is $A$.
\end{proposition}

\begin{proof}
By Theorem \ref{mainresultsec3}, $B:=A^{(\ell-1)}$ is an $n$-Auslander algebra.
Let $f$ be idempotents of $B$ whose $(i,i)$-entry is $1$ for each $1\le i\le\ell-1$ and other entries are zero. Then $\proj B\cap\inj B=\add D(Bf)$ holds.
Let $C:=fBf\simeq A^{(\ell-2)}$.
By the higher Auslander correspondence (e.g.\ \cite[Proposition 3.6]{CIM}), $Bf$ is an $n$-cluster-tilting $C$-module such that $\End_C(Bf)=B$.
Thus $C\oplus DC$ is an $n$-cluster tilting $C$-module since $\add Bf=\add(C\oplus DC)$ holds as subcategories of $\mod C$.
Moreover, $C$ is $n$-representation-finite since the inequality in Proposition \ref{Iwanaga--Gorenstein and replicated} implies $\gldim C \leq \gldim B-1=n$.
Finally the stable $n$-Auslander algebra $\underline{\End}_C(M)$ is isomorphic to $B/BfB$ (e.g.\ \cite[Lemma 3.3]{CM}), which is isomorphic to $A$.
\end{proof}

Now we are ready to prove our main result.
\begin{theorem} \label{stableendomainresult}
Let $A$ be an algebra of finite global dimension.
The following are equivalent: 
\begin{enumerate}[\rm(1)]
    \item $A$ is Iwanaga--Gorenstein and twisted fractionally Calabi--Yau.    
    \item There exists a $d$-representation-finite algebra $B$ for some $d\ge1$ such that $A$ is isomorphic to the stable $d$-Auslander algebra of $B$.
\end{enumerate}
In this case, there exists $e\ge2$ such that for each $i\ge2$, there exists an $(ei-2)$-representation-finite algebra $B_i$ such that $A$ is isomorphic to the stable $(ei-2)$-Auslander algebra of $B_i$.
\end{theorem}
\begin{proof}
(1)$\Rightarrow$(2) Let $A$ be twisted $(m,\ell)$-Calabi--Yau. By replacing $(m,\ell)$ with $(mi,\ell i)$ for $i\gg0$, we can assume $\ell\ge2$ and $n:=\ell+m-2\ge1$. By Proposition \ref{mainresultsec3+}, the condition (2) holds.
More strongly, in this case, for $e:=\ell+m$, the last assertion holds.

(2)$\Rightarrow$(1) 
This was shown in \cite[Theorem 8.7]{CDIM}.
\end{proof}

The next application gives a description of the bounded derived categories of twisted fractionally Calabi--Yau algebras in terms of $\mathbb{Z}$-graded stable module categories of higher preprojective algebras.
For a $d$-representation-finite algebra, define the \emph{$(d+1)$-preprojective algebra} $\Pi(A):=\bigoplus\limits_{i=0}^{\infty}{\Hom_A(A, \tau_d^{-i}(A))}$ with the multiplication defined as $g\cdot f:=\tau_d^{-i}(g)f$ for $f\in \Hom_A(A, \tau_d^{-i}(A))$ and $g\in \Hom_A(A, \tau_d^{-j}(A)).$
Then $\Pi(A)$ is a selfinjective twisted periodic algebra, we refer to \cite{IO} for more details. We refer for example to \cite{Ami,Ami2} for the definition of higher cluster categories ${\mathcal C}_r(A)$ associated to algebras $A$ of finite global dimension.

\begin{corollary} \label{corollaryhigherpreprojective}
Let $A$ be a twisted $(m,\ell)$-Calabi--Yau algebra of finite global dimension, and $n:=\ell+m-2$.
Let $\Pi:=\Pi(A^{(\ell-2)})$ be the $(n+1)$-preprojective algebra of $A^{(\ell-2)}$, which is an $n$-representation-finite algebra.
Then $\underline{\mod}^{\mathbb{Z}} \Pi \cong \Db(\mod A)$ and
$\underline{\mod}\,\Pi \cong {\mathcal C}_{n+1}(A)$, where
${\mathcal C}_{n+1}(A)$
 denotes the $(n+1)$-cluster category of $A$. 
 \end{corollary}

The proof of Corollary \ref{corollaryhigherpreprojective}
depends on the following result.

\begin{proposition}\cite[Theorem 4.7 and 4.15]{IO}\label{IO result}
Let $\Lambda$ be a $d$-representation-finite algebra with $d\ge1$, and $\Gamma$ the stable $d$-Auslander algebra of $\Lambda$, and $\Pi$ denotes the $(d+1)$-preprojective algebra of $\Lambda$.
Then there exist triangle equivalences $\Db(\mod \Gamma) \cong \underline{\mod}^{\mathbb{Z}} \Pi$ and $\underline{\mod}\,\Pi \cong \CC_{d+1}(\Gamma)$, the cluster category of $\Gamma$. 
\end{proposition}

 \begin{proof}[Proof of Corollary \ref{corollaryhigherpreprojective}]
Specialising Proposition \ref{IO result} to our situation, we set $\Lambda:=A^{(\ell-2)}$ and $d:=n$.
In this case $\Gamma \cong A$ by Proposition \ref{mainresultsec3+}. Thus we obtain the desired equivalences.
\end{proof}


We mention the following two folklore open problems that we formulate as conjectures here:
\begin{conjecture} \label{acyclicquiverconj}
Let $A$ be an algebra of finite global dimension.
\begin{enumerate}
    \item If $A$ is twisted fractionally Calabi--Yau, then $A$ has an acyclic Ext-quiver.
    \item If $A$ is $d$-representation-finite, then $A$ has an acyclic Ext-quiver.
\end{enumerate}
\end{conjecture}
Clearly, (1) implies (2) as every $d$-representation-finite algebra is twisted fractionally Calabi--Yau \cite[Theorem 1.1]{HI}.
Our results show that in fact both conjectures are equivalent:
\begin{proposition}
The two conjectures in Proposition \ref{acyclicquiverconj} are equivalent.
\end{proposition}
\begin{proof}
We only have to show that (2) implies (1):
Let $A$ be a twisted $\frac{m}{\ell}$-Calabi--Yau algebra. Then as we saw in the proof of Theorem \ref{stableendomainresult}, $A^{(\ell-2)}$ is a $d$-representation-finite algebra for some $d$. But in general, $A$ has an acyclic Ext-quiver if and only if $A^{(i)}$ has an acyclic Ext-quiver for some (or any) $i \geq 0$ as replicated algebras are lower triangular matrix ring algebras with $A$ on the diagonals.
\end{proof}

\subsection{Replicated algebras of Serre-formal fractionally Calabi--Yau algebras}
In this subsection, we give an explicit relationship between Theorem \ref{mainresultsec3} and \cite[Theorem 1.3]{CIM}.
We recall the following central notion of \cite{CIM}.

\begin{definition}
We call an Iwanaga-Gorenstein algebra $A$ \emph{Serre-formal} if $\nu^i(A)$ is isomorphic to a direct sum of stalk complexes for any $i\in\mathbb{Z}$ in the bounded derived category $\Db(\mod A)$ of finitely generated $A$-modules.
\end{definition}

Many important classes of algebras are Serre-formal, see \cite{CIM}. On the other hand, we introduce the following new notion for replicated algebras.


\begin{definition}\label{stalk property}
We say that a replicated algebra $A^{(m)}$ satisfies the \emph{stalk property} if for each indecomposable injective $A$-module $I$ and $1\le i\le m$, there exists a sequence $1\le\ell_1<\cdots<\ell_m$ and $X_1,\ldots,X_m\in\mod A$ such that $\Omega^{\ell_i}([I]_{m+1})=[X_i]_{m+1-i}$.
\end{definition}

The following version of Theorem \ref{mainresultsec3} gives a more explicit version of \cite[Theorem 1.3]{CIM}.

\begin{theorem}\label{Serre-formal and stalk}
Let $A$ be a finite dimensional algebra over a field $k$, and $\ell\ge2$ and $m\ge0$ be integers. Then the following are equivalent:
\begin{enumerate}[\rm(1)]
    \item $A$ is Iwanaga--Gorenstein, twisted $(m,\ell)$-Calabi--Yau and Serre-formal.
    \item $A^{(\ell-1)}$ is a minimal Auslander--Gorenstein algebra of selfinjective dimension $\ell+m-1$ satisfying the stalk property.
    \end{enumerate}
In particular, the following are equivalent:
\begin{enumerate}[\rm(1)]
    \item $A$ has finite global dimension and is twisted $(m,\ell)$-Calabi--Yau and Serre-formal.
    \item $A^{(\ell-1)}$ is a higher Auslander algebra of global dimension $\ell+m-1$ satisfying the stalk property.
    \end{enumerate}
\end{theorem}

\begin{proof}
It suffices to prove that the first two conditions are equivalent. (1)$\Rightarrow$(2) was shown in \cite[Lemma 5.11]{CIM}. In the rest, we prove (2)$\Rightarrow$(1).
Fix an indecomposable injective $A$-module $I$. Since $A^{[m]}$ is minimal Auslander-Gorenstein, the equality \eqref{proj inj} shows that $X_m$ in Definition \ref{stalk property} has to be projective as an $A$-module.
Moreover, the stalk property implies that there exists a minimal projective resolution
\[0\to Q_n\xrightarrow{f_n}\cdots\xrightarrow{f_1} Q_0\to[I]_{m+1}\to0\]
of $A^{[m+1]}$-module $[I]_{m+1}$ such that each $Q_i
$ with $0\le i\le n-1$ is a projective-injective $A^{[m]}$-module, and $\Image f_{\ell_i}=[X_i]_{m+1-i}$ for each $1\le i\le m$. This implies that $n=\ell_m$ and $Q_n=[X_m]_1$.

Let $X_0:=I$, and we claim that $\nu(X_i)\simeq X_{i+1}[\ell_{i+1}-\ell_i-1]$ holds in $\Db(\mod A)$ for each $0\le i\le m-1$, which implies that $A$ is Serre-formal.
Consider the exact sequence
\[0\to [X_{i+1}]_{m-i}\to Q_{\ell_{i+1}-1}\xrightarrow{f_{\ell_{i+1}-1}}\cdots\xrightarrow{f_{\ell_i+2}} Q_{\ell_i+1}\xrightarrow{f_{\ell_i+1}} Q_{\ell_i}\to [X_i]_{m+1-i}\to0.\]
Since this is a part of minimal projective resolution, each $Q_j$ which appears in this sequence belongs to $[DA,A]_{m-i}$. For idempotents $e:=e_{m+1-i}$ and $e':=e_{m-i}$ of $A^{[m]}$, applying $(-)e$ and $(-)e'$ gives exact sequences of $A$-modules
\[\xymatrix@R0.5em@C2em{
&0\ar[r]&Q_{\ell_{i+1}-1}e\ar[rr]^-{f_{\ell_{i+1}-1}e}&&\cdots\ar[rr]^-{f_{\ell_i+1}e}&&Q_{\ell_i}e\ar[r]&X_i\ar[r]&0\\
0\ar[r]&X_{i+1}\ar[r]&Q_{\ell_{i+1}-1}e'\ar[rr]^-{f_{\ell_{i+1}-1}e'}&&\cdots\ar[rr]^-{f_{\ell_i+1}e'}&&Q_{\ell_i}e'\ar[r]&0,
}\]
where each $Q_je$ are projective and each $Q_je'$ are injective. In particular, we have
\[\nu(X_i)\simeq\nu[Q_{\ell_{i+1}-1}e\to\cdots\to Q_{\ell_i}e]\simeq[Q_{\ell_{i+1}-1}e'\to\cdots\to Q_{\ell_i}e']\simeq X_{i+1}[\ell_{i+1}-\ell_i-1]\]
in $\Db(\mod A)$, as desired.
\end{proof}

\subsection{Calabi--Yau algebras of dimension $\frac{n}{2}$}

We collect several applications now.
Our first application gives a new inequality for relating the selfinjective dimension $\idim A$ and $m$ for a twisted $\frac{m}{\ell}$-Calabi--Yau algebra $A$. We remark that one can also give a proof without using replicated algebras, but we found this result by wondering what the inequalities in Proposition \ref{Iwanaga--Gorenstein and replicated} give in the twisted fractionally Calabi--Yau case and the proof is then quite natural and short.

\begin{proposition} \label{inequalitytheorem}
Let $\ell \geq 2$ and $A$ be a twisted $\frac{m}{\ell}$-Calabi--Yau algebra.
Then $\frac{m}{\ell-1}\leq \idim A\leq m$.

\end{proposition}
\begin{proof}
The inequality $\frac{m}{\ell-1}\leq \idim A$ can be found in \cite[2.10(c)]{HIMO}, where a proof is given for a finite global dimension but the same argument works also in the case of Iwanaga--Gorenstein algebras.
We now show $\idim A\leq m$:
By Proposition \ref{Iwanaga--Gorenstein and replicated}, we have that $\idim A + i \leq \idim A^{(i)}$ for all $i \geq 1$.
Now we use our main result Theorem \ref{mainresultsec3} to obtain $\idim A^{(\ell-1)}=m+\ell -1$, since $A$ is assumed to be twisted $\frac{m}{\ell}$-Calabi--Yau.
This gives the inequalities 
$$\idim A+\ell-1 \leq m+ \ell -1$$
which is clearly equivalent to $\idim A \leq m $. 
\end{proof}
We obtain directly the following corollary by setting $\ell =2$ in the previous Proposition \ref{inequalitytheorem}:
\begin{corollary}
Let $A$ be a twisted $\frac{m}{2}$-Calabi--Yau algebra. Then $\idim A=m.$
\end{corollary}

Twisted $\frac{m}{\ell}$-Calabi--Yau algebras with $\ell=1$ are exactly the selfinjective algebras. Next we will study the twisted $\frac{m}{\ell}$-Calabi--Yau algebras with $\ell=2$ and relate them to notions in higher Auslander--Reiten theory. 

 Recall that, for a finite dimensional algebra $A$ and $n\geq1$, we call $M\in\mod A$ \emph{$n$-precluster tilting} if the following conditions are satisfied:
\begin{enumerate}[$\bullet$]
\item $A\oplus DA\in\add M$,
\item $\Ext^i_A(M,M)=0$ for each $1\le i\le n-1$,
\item $\tau_n(M)\in\add M$ and $\tau_n^-(M)\in\add M$.
\end{enumerate}
 The Iyama--Solberg correspondence gives a bijective correspondence between precluster tilting modules and minimal Auslander--Gorenstein algebras, we refer to \cite{IyaSol} for details.

\begin{theorem} \label{onehalfCYcor}
Let $A$ be a finite dimensional Iwanaga--Gorenstein algebra. Then the following are equivalent for $n \geq 1$:
\begin{enumerate}[\rm(1)]
    \item $A$ is twisted $\frac{n}{2}$-Calabi--Yau.
    \item $\idim_AA=\idim A_A=n$, $A \oplus D A$ is an $n$-precluster tilting $A$-module and $\Hom_A(DA,A)=0$.
    \item The algebra 
    \[
\begin{pmatrix}
A & 0 \\
DA & A
\end{pmatrix}
\]
is minimal Auslander--Gorenstein of selfinjective dimension $n+1$.
\end{enumerate}
\end{theorem}

To prove this, we need the following standard fact for $n$-precluster tilting modules.

\begin{lemma}\label{ADA pCT}
Let $A$ be a finite dimensional algebra. Then the following conditions are equivalent.
\begin{enumerate}[\rm(1)]
\item $A\oplus DA$ is an $n$-precluster tilting $A$-module.
\item $\Ext^i_A(DA,A)=0$ for each $1\le i\le n-1$ and $\tau_n^-(A)\in\inj A$.
\end{enumerate}
If $A$ is basic and satisfies $\proj A\cap\inj A=\{0\}$, then the following condition is also equivalent.
\begin{enumerate}[\rm(1)]\setcounter{enumi}{2}
\item $\Ext^i_A(DA,A)=0$ for each $1\le i\le n-1$ and $\tau_n^-(A)\simeq DA$.
\end{enumerate}
\end{lemma}

\begin{proof}
This is routine, see \cite[Chapter 3]{IyaSol}.
\end{proof}

Now we are ready to prove Theorem \ref{onehalfCYcor}.

\begin{proof}[Proof of Theorem \ref{onehalfCYcor}]
(1)$\Leftrightarrow$(3): Combine Proposition \ref{CY=>AG} and Proposition \ref{AG=>CY} for $\ell=2$. 

(1)$\Leftrightarrow$(2): Without loss of generality, assume that $A$ is basic. Then the claim follows from the following chain of equivalences.
\begin{align*}
(1)\Longleftrightarrow&\  \mbox{$DA\Lotimes_ADA\simeq A[n]$ in $\Db(A)$}\\
\Longleftrightarrow&\ \mbox{$\RHom_A(DA,A)\simeq DA[-n]$ in $\Db(A)$,}\\
\Longleftrightarrow&\ \mbox{$\Ext^i_A(DA,A)$ is zero for each $i\in\Z\setminus\{n\}$, and $DA$ for $i=n$,}\\
\Longleftrightarrow&\ \mbox{$\idim A=n$, $\Ext^i_A(DA,A)=0$ for each $0\le i\le n-1$, and $\tau_n^-(A)\simeq DA$,}\\
\Longleftrightarrow&\ (2)
\end{align*}
Here, the last equivalence follows from Lemma \ref{ADA pCT}(1)$\Leftrightarrow$(3) since the condition $\Hom_A(DA,A)=0$ implies $\proj A\cap\inj A=\{0\}$.
\end{proof}

Recall that an $n$-cluster tilting $A$-module $M$ is called \emph{$\ell$-homogeneous} with $\ell\ge1$ if each $\tau_n$-orbit of indecomposable direct summands of $M$ has cardinality $\ell$. We obtain the following result, where the equivalence (1)$\Leftrightarrow$(2) improves \cite[Theorem 1.3]{HI} for the case $\ell=2$. 

\begin{corollary} \label{nhalfcorollaryfingldim}
Let $A$ be a finite dimensional algebra of finite global dimension.  Then the following are equivalent for $n \geq 1$:
\begin{enumerate}[\rm(1)]
    \item $A$ is twisted $\frac{n}{2}$-Calabi--Yau.
    \item $\gldim A=n$, and $A \oplus D A$ is a $2$-homogeneous $n$-cluster tilting $A$-module.
    \item The duplicated algebra 
    \[
\begin{pmatrix}
A & 0 \\
DA & A
\end{pmatrix}
\]
is higher Auslander algebra of global dimension $n+1$.
\end{enumerate}
\end{corollary}
The special case when a duplicated algebra is a classical Auslander algebra was already considered and fully solved by Assem and Iwanaga in \cite[Remark 1.4]{AI}.

We remark that the class of twisted $\frac{n}{2}$-Calabi--Yau algebras over a perfect field is closed under tensor products by \cite[Proposition 1.4]{HI}. We call a twisted $\frac{m}{\ell}$-Calabi--Yau algebra \emph{prime} if it is not the non-trivial tensor product of two other twisted fractionally Calabi--Yau algebras (note that again by \cite[Proposition 1.4]{HI}, fractionally Calabi--Yau algebras are closed under tensor products over perfect fields). 
We pose the following problem:
\begin{problem}
Classify prime twisted $\frac{m}{2}$-Calabi--Yau algebras of finite global dimension.
\end{problem}

\section{Examples}
In this section we will give some examples and give an outlook.

The next example shows how to use our result \ref{mainresult} to generate large new classes of higher Auslander algebras using twisted fractionally Calabi--Yau algebras.

\begin{example}
Let $Q$ be a quiver of Dynkin type. Then any finite dimensional algebra $A$ that is derived equivalent to $KQ$ is fractionally Calabi--Yau as the fractionally Calabi--Yau property is invariant under derived equivalences. 
For example any algebra of the form $A=KQ/\rad^2(KQ)$ is derived equivalent to $KQ$, see for example \cite[Theorem 1]{AMT}. 
By our main result, infinitely many of the replicated algebras $A^{(m)}$ are higher Auslander algebras. Thus from a given radical square zero algebra with Dynkin quiver, our main result gives us infinitely many higher Auslander algebras and corresponding cluster tilting modules.
The explicit quiver and relations of the replicated algebras coming from radical square zero algebras can be calculated using the recipe from \cite{S}. We give the following explicit example:
Let $A=KQ/\rad^2(KQ)$ with $Q$ the Dynkin quiver of type $D_4$:
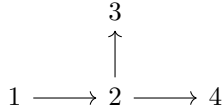
\[\begin{tikzcd}
	& 3 \\
	1 & 2 & 4
	\arrow[from=2-1, to=2-2]
	\arrow[from=2-2, to=1-2]
	\arrow[from=2-2, to=2-3]
\end{tikzcd}\]

$A^{(2)}=KQ_2/I$ is given by the quiver 
$Q_2=$
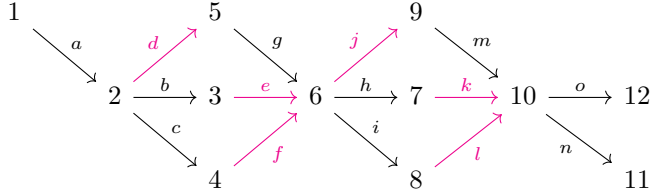
\[\begin{tikzcd}
	1 && 5 && 9 \\
	& 2 & 3 & 6 & 7 & 10 & 12 \\
	&& 4 && 8 && 11
	\arrow["a", from=1-1, to=2-2]
	\arrow["g", from=1-3, to=2-4]
	\arrow["m", from=1-5, to=2-6]
	\arrow["d", color={magenta}, from=2-2, to=1-3]
	\arrow["b", from=2-2, to=2-3]
	\arrow["c", from=2-2, to=3-3]
	\arrow["e", color={magenta}, from=2-3, to=2-4]
	\arrow["j", color={magenta}, from=2-4, to=1-5]
	\arrow["h", from=2-4, to=2-5]
	\arrow["i", from=2-4, to=3-5]
	\arrow["k", color={magenta}, from=2-5, to=2-6]
	\arrow["o", from=2-6, to=2-7]
	\arrow["n"', from=2-6, to=3-7]
	\arrow["f"', color={magenta}, from=3-3, to=2-4]
	\arrow["l"', color={magenta}, from=3-5, to=2-6]
\end{tikzcd}\]
with relations 
$$I=\langle ab,ac,gh,gi,{\color{magenta}e}i,{\color{magenta}f}h,mn,mo,{\color{magenta}k}o,{\color{magenta}l}n, e{\color{magenta}j}, {\color{magenta}fj},a{\color{magenta}d}g,b{\color{magenta}e}h, c{\color{magenta}f}i,g{\color{magenta}j}m,h{\color{magenta}k}n,i{\color{magenta}l}o, {\color{magenta}d}g-b{\color{magenta}e},b{\color{magenta}e}-c{\color{magenta}f},{\color{magenta}j}m-h{\color{magenta}k},h{\color{magenta}k}-i{\color{magenta}l} \rangle .$$

$A^{(2)}$ is a higher Auslander algebra with dominant and global dimension equal to 4.
\end{example}
In forthcoming work we plan to implement our main result into QPA in order to test whether a given algebra is twisted $\frac{m}{\ell}$-Calabi--Yau and use this for new experiments. We remark that the main result of \cite{CDIM} allows us to test whether an algebra $A$ of finite global dimension is $\frac{m}{\ell}$-Calabi--Yau using the trivial extension algebra of $A$ in case $A$ has finite global dimension, but in the infinite global dimension the results \cite{CDIM} do not apply. The classification of twisted fractionally Calabi--Yau algebras is wide open even for special classes of algebras. For example even for linear Nakayama algebras, the classification of the twisted fractionally Calabi--Yau property is wide open.

\begin{example}
The condition (2) in Theorem \ref{onehalfCYcor} means in the case $n=1$ simply that $\idim A=1$, $\Hom_A(D(A),A)=0$ and $\tau DA$ is projective. An explicit example of an Iwanaga--Gorenstein algebra $A$ with $\idim A=1$ (and infinite global dimension) satisfying $\Hom_A(DA,A)=0$ and $\tau DA$ being projective is given by $A=KQ/I$ with $Q=$
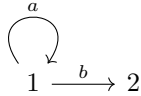
\[\begin{tikzcd}
	1 & 2
	\arrow["a", from=1-1, to=1-1, loop, in=55, out=125, distance=10mm]
	\arrow["b", from=1-1, to=1-2]
\end{tikzcd}\]
and $I=\langle a^2 \rangle$.
$A^{(1)}=KQ_2/I_2$ is given by the following quiver $Q_2$:
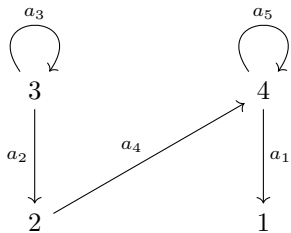
\[\begin{tikzcd}
	3 &&& 4 \\
	\\
	2 &&& 1
	\arrow["{a_3}", from=1-1, to=1-1, loop, in=55, out=125, distance=10mm]
	\arrow["{a_2}"', from=1-1, to=3-1]
	\arrow["{a_5}", from=1-4, to=1-4, loop, in=55, out=125, distance=10mm]
	\arrow["{a_1}", from=1-4, to=3-4]
	\arrow["{a_4}", from=3-1, to=1-4]
\end{tikzcd}\]

with relations $I_2=\langle a_4 a_1, a_3^2,a_5^2,a_2 a_4 a_5-a_3 a_2 a_4 \rangle.$ $A^{(1)}$ is minimal Auslander--Gorenstein with $\idim A^{(1)}=\domdim A^{(1)}=2$.
Thus, $A$ is twisted $\frac{1}{2}$-Calabi--Yau by Theorem \ref{onehalfCYcor}.
\end{example}
It would be interesting to find a general classification of Iwanaga--Gorenstein algebras that are twisted $\frac{1}{2}$-Calabi--Yau. Another example of a twisted $\frac{1}{2}$-Calabi--Yau (but with finite global dimension) is the path algebra $KQ$, where $Q$ is non-linear oriented of Dynkin type $A_3$.
In fact this is the only example of a ring-indecomposable $K$-algebra over an algebraically closed field $K$ in the finite global dimension case, see \cite[Remark 1.4]{AI}.

We pose the following problem:
\begin{problem}
Let $A$ be a fractionally Calabi-Yau algebra of finite global dimension. Define 
$$d_A:= \inf \{ d \geq 1 \mid A \cong \underline{\End}_B(M), \ \text{where $M$ is $d$-cluster tilting and $\gldim B=d$} \}.$$
Compute $d_A$ explicitly in terms of homological information on $A$ and/or find explicit upper bounds.
\end{problem}
In our article we showed in Theorem \ref{stableendomainresult} that $d_A \leq m + \ell-2$ when $A$ is $\frac{m}{\ell}$-Calabi-Yau of finite global dimension. However, this bound is not optimal in general as the next example shows:
\begin{example}
Let $A=KQ$ with $Q$ of linear oriented Dynkin type $A_3$. Then $A$ is $\frac{2}{4}$-Calabi-Yau. We have $m + \ell-2=4$ and $d_A=2$, since $A$ is isomorphic to $\underline{\End}_B(M)$, where $B=KQ/I$ is the 2-representation-finite algebra with quiver
\[Q=[\begin{tikzcd}
	1 & 2 & 3 & 4 & 5
	\arrow["{a_1}", from=1-1, to=1-2]
	\arrow["{a_2}", from=1-2, to=1-3]
	\arrow["{a_3}", from=1-3, to=1-4]
	\arrow["{a_4}", from=1-4, to=1-5]
\end{tikzcd}]\]
and relations $I=\langle a_1 a_2 a_3 a_4 \rangle$ and $M$ is the 2-cluster tilting module $M=D(A) \oplus \tau_2(D(A))$.
\end{example}





\begin{thebibliography}{Gus}

\bibitem[A1]{Ami}
Amiot, C.: {\it Cluster categories for algebras of global dimension 2 and quivers with potential.}
Ann. Inst. Fourier (Grenoble) 59 (2009), no. 6, 2525--2590.

\bibitem[A2]{Ami2} Amiot, C.: {\it On generalized cluster categories.}
In: Skowro\'nski, A., Yamagata, K. (eds.) {\it Representations of Algebras and Related Topics}
(EMS Ser. Congr. Rep.), Eur. Math. Soc., Z\"urich, pp.~1--53 (2011).

\bibitem[AMT]{AMT} Aquino, R. M.; Marcos, E. N.; Trepode, S.: {\it  On the existence of a derived equivalence between a Koszul algebra and its Yoneda algebra.}  J. Algebra Appl. 13, No. 4, (2014). 

\bibitem[ABST]{ABST} Assem, I.; Br\"{u}stle, T.; Schiffler, R.; Todorov, G.:
\emph{{$m$}-cluster categories and {$m$}-replicated algebras}.
J. Pure Appl. Algebra 212 (2008), 884--901.

\bibitem[ABST2]{ABST2} Assem, I.; Br\"{u}stle, T.; Schiffler, R.; Todorov, G.: {\it Cluster categories and duplicated algebras.} J. Algebra 305, No. 1, 548-561 (2006).


\bibitem[AI]{AI} Assem, I.; Iwanaga, Y.: {\it On a class of representation-finite QF-3 algebras.} Tsukuba J. Math. 11, 199-217 (1987). 








\bibitem[ARS]{ARS} Auslander, M.; Reiten, I.; Smalo, S.: {\it Representation Theory of Artin Algebras} Cambridge Studies in Advanced Mathematics, 36. Cambridge University Press, Cambridge, 1997. xiv+425 pp.

\bibitem[AR1]{AR1}
Auslander, M.; Reiten, I.:{\it Stable equivalence of dualizing $R$-varieties.}
Adv. Math. 12 (1974), 306--366.

\bibitem[AR2]{AR2}
Auslander, M.; Reiten, I.:{\it Stable equivalence of dualizing $R$-varieties. II. Hereditary dualizing $R$-varieties.}
Adv. Math. 17 (1975), no. 2, 93--121.

\bibitem[AR3]{AR3}
Auslander, M.; Reiten, I.:{\it Stable equivalence of dualizing $R$-varieties. III. Dualizing $R$-varieties stably equivalent to hereditary dualizing $R$-varieties.}
Adv. Math. 17 (1975), no. 2, 122--142.

\bibitem[AR4]{AR4}
Auslander, M.; Reiten, I.:{\it Stable equivalence of dualizing $R$-varieties. IV. Higher global dimension.}
Adv. Math. 17 (1975), no. 2, 143--166.

\bibitem[AR5]{AR5}
Auslander, M.; Reiten, I.:{\it Stable equivalence of dualizing $R$-varieties. V. Artin algebras stably equivalent to hereditary algebras.}
Adv. Math. 17 (1975), no. 2, 167--195.

\bibitem[BS]{BS} Berggren, J.; Serhiyenko, K.: {\it Classical tilting and $\tau$-tilting theory via duplicated algebras.} https://arxiv.org/abs/2512.13893.

\bibitem[BMRRT]{BMRRT}
Buan, A.~B., Marsh, R.~J., Reineke, M., Reiten, I., and Todorov, G.:
\emph{Tilting theory and cluster combinatorics},
Adv.\ Math.\ \textbf{204} (2006), no.~2, 572--618.

\bibitem[BMR]{BMR}
Buan, A.~B., Marsh, R.~J., and Reiten, I.:
\emph{Cluster-tilted algebras},
Trans.\ Amer.\ Math.\ Soc.\ \textbf{359} (2007), no.~1, 323--332.

\bibitem[CIM]{CIM} Chan, A.; Iyama, O.; Marczinzik, R.: {\it Auslander--Gorenstein algebras from Serre-formal algebras via replication.} Adv. Math. 345, 222-262 (2019).

\bibitem[CDIM]{CDIM} Chan, A.; Darp\"o, E.; Iyama, O.; Marczinzik, R.:{\it Periodic trivial extension algebras and fractionally Calabi--Yau algebras.} Ann. Sci. \'Ec. Norm. Sup\'er. (4) 58 (2025), no.~2, 463--510.

\bibitem[Cha]{Chap} Chapoton, F.: {\it Posets and fractional Calabi--Yau categories. } Math. Res. Rep. (Amst.) 6, 1-16 (2025).


\bibitem[CC]{CC} Caorsi, M.; Cecotti, S.:{\it Homological classification of $4d \mathcal{N}$= 2 QFT. Rank-1 revisited},  J. High Energ. Phys. 2019, 13 (2019).

\bibitem[Che]{Che} Chen, X.: {\it Gorenstein Homological Algebra of Artin Algebras.} https://arxiv.org/abs/1712.04587.

\bibitem[CM]{CM} Cruz, T.; Marczinzik, R.: {\it An Auslander--Buchsbaum formula for higher Auslander algebras and applications.} https://arxiv.org/abs/2502.08422.

\bibitem[DI]{DI} Darp\"o, E.; Iyama, O.: {\it $d$-representation-finite self-injective algebras.} Adv. Math. 362 (2020), 106932.

\bibitem[DJL]{DJL} Dyckerhoff, T.; Jasso, G.; Lekili, Y.: {\it The symplectic geometry of higher Auslander algebras: symmetric products of disks. }
Forum Math. Sigma 9, Paper No. e10, 49 p. (2021). 

\bibitem[FK]{FK} Favero, D.;Kelly, T. L.:{\it Fractional Calabi--Yau categories from Landau--Ginzburg models.} Algebraic Geometry , 596-649, 2018.

\bibitem[FGR]{FGR} Fossum, R.; Griffith, P.; Reiten, I.: {\it Trivial extensions of abelian categories. Homological algebra of trivial extensions of abelian categories with applications to ring theory.} Lecture Notes in Mathematics, Vol. 456. Springer Verlag, Berlin-New York, 1975


\bibitem[G]{G} Gottesman, T: {\it Fractionally Calabi--Yau lattices that tilt to higher Auslander algebras of type A.} https://arxiv.org/abs/2406.09148.

\bibitem[H]{H} Happel, D.: {\it Triangulated categories in the representation theory of finite dimensional algebras.} London Mathematical Society Lecture Note Series 119, 1988. 

\bibitem[HI]{HI} Herschend, M.; Iyama, O.: {\it $n$-representation-finite algebras and twisted fractionally Calabi--Yau algebras.} Bull. Lond. Math. Soc. 43 (2011), no.~3, 449--466.

\bibitem[HIMO]{HIMO} Herschend, M.;Iyama, O.;Minamoto, H.;Oppermann, S.: \emph{Representation theory of Geigle--Lenzing complete intersections}, Mem. Amer. Math. Soc.~285 (2023), no.~1412, vii+141 pp.


\bibitem[I1]{I1} Iyama, O.: {\it Higher-dimensional Auslander--Reiten theory on maximal orthogonal subcategories.} Adv. Math. 210 (2007), no. 1, 22--50.

\bibitem[I2]{I2} Iyama, O.: {\it Cluster tilting for higher Auslander algebras.} Advances in Mathematics Volume 226, Issue 1, 15 January 2011, Pages 1-61.


\bibitem[IO]{IO} Iyama, O.; Oppermann, S.: {\it Stable categories of higher preprojective algebras.} Advances in Mathematics
Volume 244, 10 September 2013, Pages 23-68.

\bibitem[IS]{IyaSol} Iyama, O; Solberg, {{\O}}.: {\it Auslander--Gorenstein algebras and precluster tilting.} Advances in Mathematics Volume 326, 21 February 2018, Pages 200-240.

\bibitem[IY]{IY}
Iyama, O. and Yoshino, Y.:
\emph{Mutation in triangulated categories and rigid {C}ohen--{M}acaulay modules},
Invent.\ Math.\ \textbf{172} (2008), no.~1, 117--168.

\bibitem[J]{J} Jasso, G.: {\it n-abelian and n-exact categories.} Math. Z., Volume 283 (2016), Issue 3-4, 703-759.

\bibitem[Ke1]{Ke1} Keller, B.: \emph{Calabi--Yau triangulated categories.}  EMS Series of Congress Reports, 467-489 (2008). 


\bibitem[Ke2]{Ke2} Keller, B.: \emph{The periodicity conjecture for pairs of Dynkin diagrams}, Ann. of Math. (2) 177 (2013), no. 1, 111--170.

\bibitem[Ke3]{Ke3} Keller, B.: {\it On triangulated orbit categories.} 
Doc. Math. 10, 551-581 (2005).

\bibitem[Kl]{Kle} Kleinau, M. with an appendix by R. Marczinzik: \emph{Cambrian lattices are fractionally Calabi--Yau via 2-cluster
combinatorics.} https://arxiv.org/abs/2603.23354

\bibitem[Ko]{Kon} Kontsevich, M.: {\it
Cours donn\'e \`a l'\`ecole normale sup\'erieure.}
\url{https://www.lamfa.u-picardie.fr/marin/docs/coursKont.pdf}, 1998.

\bibitem[Kr]{Kr} Krause, H.: {\it Homological theory of representations.} Cambridge Studies in Advanced Mathematics 195.  Cambridge University Press. xxxiv, 482 p. (2022).


\bibitem[KLM1]{KLM1} Kussin, D.;Lenzing, H.; Meltzer, H.: {\it Triangle singularities, ADE-chains, and weighted projective lines.} Advances in Mathematics, Volume 237, Pages 194-251, 2013.

\bibitem[KLM2]{KLM2}
Kussin, D.; Lenzing, H.; Meltzer, H.:
{\it Nilpotent operators and weighted projective lines.}
Journal f\"ur die reine und angewandte Mathematik (Crelle's Journal),
Volume 685, Pages 33--71, 2013.

\bibitem[Ku]{K}  Kuznetsov, A.: \emph{Calabi--Yau and fractional Calabi--Yau categories.} Journal f{\"u}r die reine und angewandte Mathematik, Volume 2019, Issue 753.



\bibitem[MY]{MY} Miyachi, J.;Yekutieli, A.: {\it Derived Picard groups of finite-dimensional hereditary algebras.} Compositio Math.
129 (2001), no. 3, 341-368.

\bibitem[OT]{OT} Oppermann, S.; Thomas, H.: {\it Higher-dimensional cluster combinatorics and representation theory.} J. Eur. Math. Soc. (JEMS) 14, No. 6, 1679-1737 (2012). 

\bibitem[PR]{PR} Palmer, I.; Roos, J.:{\it Explicit formulae for the global homological dimensions
of trivial extensions of rings.} J. Algebra 27 (1973), 380-413.



\bibitem[QPA]{QPA} The QPA-team, QPA - Quivers, path algebras and representations - a GAP package, Version 1.33; 2022 \url{https://folk.ntnu.no/oyvinso/QPA/}.

\bibitem[Re]{Re}
Reiten, I.:{\it Stable equivalence of dualizing $R$-varieties. VI. Nakayama dualizing $R$-varieties.}
Adv. Math. 17 (1975), no. 2, 196--211.

\bibitem[Ro]{R} Rognerud, B.:
 {\it The bounded derived categories of the Tamari lattices are fractionally Calabi--Yau.} Adv. Math. 389,  31 p. (2021).



\bibitem[Sc]{S} Schr\"oer, J.: {\it On the quiver with relations of a repetitive algebra.}  Arch. Math. 72, No. 6, 426-432 (1999). 

\bibitem[Se]{Se} Sen, E.: {\it Higher Auslander Algebras arising from Dynkin Quivers and n-Representation Finite Algebras.} arXiv:2307.13262

\bibitem[We]{W} Weng, W.: {\it A recollement approach to Brieskorn--Pham singularities.} https://arxiv.org/abs/2512.09692.

\bibitem[Wi]{Wi} Williams, N.: {\it The two higher Stasheff--Tamari orders are equal.} J. Eur. Math. Soc. (2024) DOI 10.4171/JEMS/1497

\bibitem[X]{X} Xing, W.: {\it Replicated algebras derived equivalent to higher Auslander algebras of type A.} https://arxiv.org/abs/2511.22655.

\bibitem[XZ]{XZ} Xiong, B.; Zhang, P.: {\it Gorenstein-projective modules over triangular matrix Artin algebras.}  J. Algebra Appl. 11, No. 4, (2012). 

\bibitem[Y]{Ye} Yekutieli, A.: \emph{Derived categories}, Cambridge Studies in Advanced Mathematics, 183. Cambridge University Press, Cambridge, 2020.

\end{thebibliography}
\end{document}